\documentclass[11 pt]{article}
\usepackage{amsmath,amscd}
\usepackage{amstext}
\usepackage{amssymb,epsfig}
\usepackage{amsthm}
\usepackage{amsfonts}
\usepackage{fancybox, floatrow}
\usepackage{graphicx}
\usepackage[all]{xy}
\usepackage{subeqnarray}

\usepackage[utf8]{inputenc}      
\usepackage{tikz}

\newtheorem{mainth}{Theorem}
\newtheorem{proposition}{Proposition}[section]
\newtheorem{theorem}[proposition]{Theorem}

\newtheorem{lemma}[proposition]{Lemma}
\newtheorem{corollary}[proposition]{Corollary}

\newtheorem{definition}[proposition]{Definition}

\def\eg{{\em e.g.\ }}

\def\N{{\mathbb N}}

\def\R{{\mathbb R}}

\def\D{{\mathbb D}}
\def\K{{\mathbb K}}

\newcommand {\CA}{{\mathcal A}}
\newcommand {\CB}{{\mathcal B}}

\newcommand {\CL}{{\mathcal L}}
\newcommand {\CM}{{\mathcal M}}

\newcommand {\CR}{{\mathcal R}}

\newcommand {\CW}{{\mathcal W}}

\def\s{\sigma}
\def\l{\lambda}

\def\1{ {\hbox{{\it 1}} \!\! I} }

\def\al{\alpha}

\def\de{\delta}

\def\om{\omega}

\def\8{\infty}

\newcommand{\ol}{\overline}
\def\disp{}

\usepackage{float,color}

\newcommand{\wh}{\widehat}

\def\ninf{{n\to+\8}}
\def\minf{{m\to+\8}}

\theoremstyle{definition}
\newtheorem{remark}{Remark}

\begin{document}
\title{Thermodynamic formalism and substitutions: Renormalization operator.}
\author{
Nicolas Bédaride\footnote{
    Aix Marseille Université, CNRS,
    Centrale Marseille, I2M, UMR 7373, 13453 Marseille, France.
    Email: \texttt{nicolas.bedaride@univ-amu.fr}} 
    \and 
    Pascal Hubert\footnote{
    Aix Marseille Université, CNRS,
    Centrale Marseille, I2M, UMR 7373, 13453 Marseille, France.
    Email: \texttt{pascal.hubert@univ-amu.fr}}
    \and 
    Renaud Leplaideur\thanks{R. Leplaideur thanks Frumam for kind hospitality.} 
    \footnote{Université de la Nouvelle-Calédonie 145, Avenue James Cook - BP R4 98 851 -
Nouméa Cedex Nouvelle Calédonie
    \texttt{renaud.leplaideur@unc.nc}}
}
\date{}
\maketitle

\begin{abstract}
This paper studies properties of a Renormalization Operator for potentials in symbolic dynamics. These operators first appeared in \cite{BLL} and the link with substitutions was done in \cite{BL1}. Their fixed points are natural candidates to have pathologic behavior such as phase transitions. If $R$ is such an operator, we study the convergence of $R^{n}(\varphi)$ to the non-nul fixed point.

We define the family of marked substitutions, which contains the Thue-Morse substitution, and show that the associated renormalization operators on potentials admits a unique non-nul continuous fixed point. Then, we show that $R^{n}(\varphi)$ converges to the fixed point as soon as $\varphi$ has the right germ close to $\K$.


\end{abstract}

\section{Introduction}
\subsection{Background}
This paper studies Thermodynamic Formalism with ``pathologic'' behaviors. More precisely it carries out the investigations of renormalization for potentials for symbolic dynamics and its links with substitutions, subject tackled in \cite{BLL,BL1}.

We remind that the Thermodynamic Formalism has been introduced in the 70's  for the Uniformly Hyperbolic Dynamical Systems, mainly by Sinai, Ruelle and Bowen (see \cite{sinai,ruelle,Bow.08}). A variational Principle is associated to a function called a \emph{potential}  and allows to singularize an invariant probability called \emph{equilibrium state}. Then, the system is studied with respect to this invariant measure. 

Since the 80's, it has been a big challenge to extend the Thermodynamic Formalism to systems with weaker hyperbolicity. In spirit, it is expected, on the one hand to get existence and uniqueness of the equilibrium state and on the other hand that the pressure function associated to that potential is analytic. We point out that this challenge is still not yet completed.

On the other hand, and for several years now, researches has been done in a transversal direction: the idea is to benefit the weaker hyperbolicity to exhibit behaviors that cannot occur for the uniformly hyperbolic systems (see \eg \cite{letelier,iommi-todd,leplaideur-butterfly,lopes,Pesin-Zhang,sarig06}). This is what we refer to \emph{pathologic behaviors}. The main historical examples are  the Manneville-Pomeau case and the Hofbauer potential which both exhibit a phase transition (see \cite{MP,hofbauer}).  


Actually, the consequence of Ruelle's Theorem is that one needs to release conditions on the hyperbolicity of the system or on the regularity for the potential, to exhibit systems with pathologic behaviors. Indeed, the Manneville-Pomeau case deals with non-uniform expansivity and  the Hofbauer Potential is non-H\"older continuous.

 One of the main results in \cite{BLL} is to make a connection between these two historical examples. This shows that the two ways to exhibit pathologic behaviors are not necessarily disconnected. The connection comes from a \emph{renormalization operator} defined on potentials.  It was shown that for the Manneville-Pomeau map $f:[0,1]\to [0,1]$ the natural potential (namely $-\log f'$) is a fixed point for this renormalization operator. As the dynamics of $f$ is congugate to the 2-full shift, it was proved that the Hofbauer potential actually was the fixed point for the natural renormalization operator for potential in the symbolic dynamics. 

Later, it was shown in \cite{BL1} that renormalization operators for potentials were a way to exhibit non-uniformly expanding maps with a more complicated set of singular points than for the Manneville-Pomeau example, and still with phase transition for the natural potential. This was done  only for the Thue-Morse substitution but this has been the starting point to make a link between renormalization for potentials and substitutions.

The present paper proves a general statement in the study of fixed point for renormalization for potentials associated to substitutions, instead of the study of one example as in \cite{BL1}. We define the class of \emph{marked} substitutions. Roughly speaking it means that if $a$ is a digit and $H$ the substitution, it is sufficient to know the first\footnote{left marked} or the last\footnote{right marked} letter of $H(a)$ to know what $a$ is. Clearly,  the Thue-Morse substitution which is defined by $H(0)=01$ and $H(1)=10$ is left and right marked. In this paper, it is proved that for left and right marked substitutions, the renormalization operator $\CR$ admits a unique continuous and non-nul fixed point. Conditions on the potentials $\varphi$ are given to insure that $R^n(\varphi)$ converges to the fixed point. 
We point out that after that the first version of this work has been announced, J. Emme managed to get a similar result for the $k$-bonacci case, which are right-marked  but not left-marked substitutions (see \cite{Emme}).

 Finally, we highlight that beyond participating to the inquires to exhibit potentials with pathologic behaviors, the problem of renormalization for potentials with respect to substitutions has shown  a new way to study substitutions, that is from \emph{outside} instead of from\emph{ inside}.
Substitutions are quasi-periodic systems. They have zero entropy and are not chaotic in the sense that the past almost predicts the future. The unique small chaotic behaviors occur at \emph{left or right or bi-special words}. These are properties on the language of the substitution, and this is what from \emph{inside} means. 
The point is that the small chaotic behaviors from inside  generate \emph{accidents\footnote{See below the exact definition.}} outside which may disturb how the sequence $R^n(\varphi)$ converges to the fixed point. It turns out (see Prop. \ref{prop:acc}) that if the substitution is marked, then the accidents occur at fixed moments that are all obtained by  a rescaling procedure and finally do not  disturb the convergence. 

\subsection{Results}
Let $\mathcal{A}$ be a finite set called the alphabet with cardinality $D\ge 2$. Elements of $\CA$ are called \emph{letters} or \emph{digits}.
A word is a finite  or infinite string of digits. If $v$ is the  finite word $v=v_0\dots v_{n-1}$ then $n$ is called the length of the word $v$ and is denoted by $|v|$. The set of all finite words over $\mathcal{A}$ is denoted by $\mathcal{A}^*$. 

If $u=u_{0}\ldots u_{n-1}$ is a finite word and $v=v_{0}\ldots$ is a word, the concatenation $uv$ is the new word $u_{0}\ldots u_{n-1}v_{0}\ldots$. If $v$ is a finite word, $v^{n}$ denotes the concatenated word 
$$v^{n}=\underbrace{v\ldots v}_{n\text{ times}}.$$ 

If $u=u_{0}\ldots u_{n-1}$ is a word, a prefix of $u$ is any factor $u_{0}\ldots u_{j}$ with $j\le n-1$. A suffix of $u$ is any word of the form $u_{j}\ldots u_{n-1}$ with $0\le j\le n-1$.  

The shift map is the map defined on $\mathcal{A}^\mathbb{N}$ by $\sigma(u)=v$ with $v_n=u_{n+1}$ for all integer $n$.
We endow $\mathcal A$ with the discrete topology and consider the product topology on $\mathcal A^\N$. This topology is compatible with the distance $d$ on $\mathcal A^{\mathbb N}$ defined by
$$d(x,y)=\frac{1}{D^n}\quad \text{if}\quad n=\min\{i\geq 0,  x_i\neq y_i\}.$$
\begin{definition}
An infinite word $u$ is said to be periodic (for $\s$) if it is the infinite concatenation of a finite word $v$, that is $u=vvvv\ldots$
In that case we set  $u=v^\8$.
\end{definition}

A substitution $H$ is a map from an alphabet $\mathcal{A}$ to the set $\mathcal{A}^*\setminus\{\epsilon\}$ of nonempty finite words on $\mathcal{A}$. It extends to a morphism of $\mathcal{A}^*$ by concatenation, that is $H(uv)=H(u)H(v)$. 

Several basic notions on substitutions are recalled in Section \ref{sec-substi}. We also refer to  \cite{Pyth.02}.  We recall here the notions we need to state our results. 

\begin{definition}
\label{def-basicsubsti}
If $H$ is a substitution, its \emph{incidence matrix} is the $D\times D$ matrix $\CM_{H}$ with entries $a_{ij}$ where $a_{ij}$ is the number of $j$'s in $H(i)$. Then, $H$ is said to be \emph{primitive} if all entries of $\CM_{H}^{k}$ are positive for some $k\ge 1$. 

A $k$-periodic point of $H$ is an infinite word $u$ with $H^k(u)=u$ for some $k>0$. If $k=1$ the point is said to be fixed. 
Then, $H$ is said to be \emph{aperiodic} if no fixed point for $H$ is a periodic sequence for $\s$. 
\end{definition}
We point out an equivalent definition for being primitive. The substitution $H$ is primitive if and only if there exists an integer $k$ such that  for every couple of letters $(i,j)$, $j$ appears in $H^{k}(i)$.

\medskip
Let $H$ be a substitution over the alphabet $\mathcal{A}$, and $a$ be a letter such that $H(a)$ begins with $a$ and $|H(a)|\geq 2$. Then there exists a fixed point $u$ of $H$ beginning with $a$ (see  \cite[1.2.6]{Pyth.02}). This infinite word is the limit of the sequence of finite words $H^n(a)$. Assume that $\om$ is a fixed point for $H$, then we set 
$$\K:=\ol{\left\{\s^{n}(\om),\ n\in\N\right\}}.$$
If $H$ is a primitive  substitution, then $\K$ does not depend on the fixed point $\om$. It is called the {\bf subshift} associated to the substitution. If $H$ is aperiodic, then $\K$ is uniquely ergodic but not reduced to a  $\s$-periodic orbit. In that case, the unique $\s$-invariant probability is  denoted by $\mu_{\K}$

We recall that the  \emph{language of a primitive substitution} is the set of finite words which appear in a fixed point. It is denoted by $\mathcal L_H$.
\begin{definition}
\label{def-2full}
 A substitution is said to be \emph{2-full} if any word of length 2 in $\CA^{*}$ belongs to the language of the substitution. 
 A substitution is said to be {\bf marked} if the set of the first (and last) letters of the images of the letters by the substitution is in bijection with the alphabet.
\end{definition}

\begin{definition}
Let $n$ be a positive integer. For $x\in \CA^{\N}$ of the form $x=a\ldots$  and for a primitive, 2-full and marked substitution $H$, 
we set $t_{n}(x)=|H^{n}(a)|$.
\end{definition}
Let us define $R$ by: 

\begin{equation}
\label{eq-def-op-R}
\begin{array}{rcl}
R:\mathcal C(\mathcal  A^\mathbb N,\mathbb{R})&\rightarrow&\mathcal C(\mathcal A^\mathbb N,\mathbb{R})\\
\varphi(x)&\mapsto& R(\varphi)(x)=\displaystyle\sum_{i=0}^{t_1(x)-1}\varphi\circ\sigma^i\circ H(x)
\end{array}
\end{equation}

Then we have:
\begin{mainth}\label{thm:fixe}
Let $H$ be a 2-full, marked, aperiodic and primitive substitution, then there exists $U:\CA^{\N}\to\R$ continuous such that $R(U)=U$. 

Consider a map $\varphi:\mathcal{A} ^{\mathbb N}\rightarrow \mathbb{R}$  such that $\varphi_{|\K}\equiv 0$ and $\varphi(x)=\disp\frac{\disp g(x)}{\disp p^\alpha}+\frac{\disp h(x)}{\disp p^\alpha}$ if $d(x,\K)=D^{-p}$, where $g$ is a continuous positive function and $h$ is continuous and satisfies $h_{|\K}\equiv 0$. 

Then, for every $x$ in $\CA^{\N}$ we have 
$$\lim_{m\to+\8} R^m\varphi(x)=\begin{cases}0 \quad& \text{ if }\alpha>1,\\ 
+\infty \quad &\text{ if }\alpha<1,\\
\disp\int g\,d\mu_{\K} .\, U(x) \quad &\text{ if }\alpha=1.
\end{cases}$$ 
\end{mainth}
\begin{remark}
The expression of $U$ is explicit for a given substitution. It will be explained during the proof, and in Section \ref{sec:tm}.
\end{remark}

In the following, we denote by $\Xi_{\al}$ the set of potentials $V=-\varphi$ of the form $\varphi(x)=\frac{g(x)}{p^\alpha}+\frac{h(x)}{p^\alpha}$ as in Theorem \ref{thm:fixe}. 

We emphasize that the Thue-Morse substitution is 2-full, marked, aperiodic and primitive. Therefore, Theorem \ref{thm:fixe} improves \cite{BL1} where  only the Cesaro-convergence was proved. 
 
\subsection{Outline of the paper}
First of all in Section \ref{sec-substi} we recall some classical definitions and results on substitutions and symbolic dynamics. 
The last part of this section is devoted to some background on the notion of accidents, defined in \cite{BL1}. 

Then in Section \ref{secthm1} we prove Theorem \ref{thm:fixe}. The proof is decomposed in several parts. 
We obtain a formula for $R^m\varphi$ in Lemma \ref{lem:fixecalc}. To study the convergence of this term we need to get good estimates for $\de_{i}^{n}(x)$ (defined in Subsection \ref{sec:acc}) for $i< t_{n}(x)$ and for any $x\notin\K$. This is done in Corollary \ref{cor-accidentHkx}. Finally we compute the limit in two steps: one for the simplest case $g\equiv 1$ and one for the general case, see Subsection \ref{fin-thm1}. 

In Section \ref{sec:tm} we give a concrete proof of Theorem \ref{thm:fixe} for the example of the Thue-Morse subshift.
\section{More definitions and tools}\label{sec-substi}

\subsection{Words, languages and special words}

For this paragraph we refer to \cite{Pyth.02}.

\begin{definition}
A word $v=v_0\dots v_{r-1}$ is said to occur at position $m$ in an infinite word $u$ if there exists an integer $m$ such that for all $i\in[0;r-1]$ we have $u_{m+i}=v_i$. We say that the word $v$ is a \emph{factor} of $u$. 

For an infinite word $u$, the language of $u$ (respectively the language of length $n$)  is the set of all words (respectively all words of length $n$) in $\mathcal{A}^*$ which appear in $u$. We denote it by $\mathcal L(u)$ (respectively $\mathcal L_{n}(u)$). 
Then, the sequence of finite languages $(\CL_{n}(u))_{n\in\mathbb N}$ is said to be the factorial language for $\CL(u)$. 

\end{definition}

\begin{definition}\label{def-rec-lin} \cite[Sec7]{DHS99}.
The dynamical system associated to an infinite word $u$ is the system $(\mathbb K_u,\sigma)$ where $\sigma$ is the shift map and $\mathbb K_u=\overline{\{\sigma^n(u),n\in\mathbb{N}\}}$.
An infinite word $u$ is said to be recurrent if every factor of $u$ occurs infinitely often. 
\end{definition}
Remark that $u$ is recurrent is equivalent to the fact that $\sigma$ is onto on $\mathbb  K_u$. Moreover we have equivalence between $\omega\in \mathbb K_u$ and $\mathcal L_\omega\subset \mathcal L_u$. Thus the language of $\mathbb K_u$ is equal to the language of $u$.

\begin{definition}\label{def-bispe}
Let $\mathcal L=(\mathcal L_n)_{n\in\mathbb{N}}$ be a factorial and extendable language. The \emph{ complexity function} $p:\N\to\N$
 is the function defined by 
$p(n)\!:=card( \CL_n)$. For $v \in \mathcal L_n$ let us define
\begin{align*}
m_{l}(v)  & = card\{a\in \mathcal{A}, av\in \mathcal L_{n+1}\},\\
m_{r}(v)  & = card\{b\in \mathcal{A},vb\in \mathcal L_{n+1}\},\\
m_{b}(v) & = card\{(a,b)\in \mathcal{A}^2, avb\in \mathcal L_{n+2}\},\\
i(v)           & = m_b(v)-m_r(v)-m_l(v)+1.\\
\end{align*}

\begin{itemize}
\item A word $v$ is called right special if $m_{r}(v)\geq 2$. 
\item A word $v$ is called left special if $m_{l}(v)\geq 2$. 
\item A word $v$ is called bispecial if it is right and left special.
\end{itemize}
\end{definition}

\begin{definition}
A word $v$ such that $i(v)< 0$ is called a weak bispecial. 
A word such that $i(v)>0$ is called a strong bispecial.
A bispecial word $v$ such that $i(v)=0$ is called a neutral bispecial.
\end{definition}

\subsection{Substitutions}

\subsubsection{Some more definitions}

\begin{definition}
Let $H$ be a substitution. The set of all prefixes and all suffixes  for all the $H(a)$, $a\in\CA$, are respectively denoted by $\mathcal P$ and $\mathcal S$.
\end{definition}

For a substitution $H$, we recall that its language is denoted by $\mathcal L_H$.

\begin{definition}
Let $H$ be a substitution. 
We say that the word $u\in\mathcal{L}_H$ is {\bf uni desubstituable} if there exists only one way to write 
$u=sH(v)p$ with $p\in\mathcal P,s\in\mathcal S$ where 
\begin{enumerate}
\item $p$ is a prefix of $H(\wh p)$ for some $\wh p$,
\item $s$ is a suffix of $H(\wh s)$ for some $\wh s$,
\item $\wh s v\wh p$ is a word in $\CL_{H}$.
\end{enumerate}
 \end{definition}

We recall the following theorem
\begin{theorem}\cite{Mosse.92}\label{thm:mosse}
Let $H$ be a marked, primitive, aperiodic substitution.

There exists a constant $N_H$ such that for every word $w\in \mathcal L_H$ the word $w^{N_H}$ does not belong to this language.

\end{theorem}
\begin{remark}
Remark that $N_H$ can be computed by an algorithm. 
\end{remark}

\subsubsection{Length of words in the language of a substitution}
If $H$ is a primitive substitution, the Perron Frobenius theorem shows that  the incidence matrix admits a single and simple dominating eigenvalue. We denote it by $\l$. It is a positive real number. The rest of the spectrum is strictly included into the disc $\D(0,\lambda)$.

Then, we emphasize that  there exists a constant $K$ such that the length of a word $H^n(v)$ satisfies
\begin{equation}
\label{eq-length-hn}
|H^n(v)|\leq K\lambda^n.
\end{equation}

Thus in all the following computations we will consider this upper bound.

\subsection{Accidents}\label{sec:acc}
Let $\mathbb K$ be the subshift associated to the substitution $H$.
Let $x$ be an element of $\mathcal{A}^\mathbb N$ which does not belong to $\mathbb K$, then we define and denote:
\begin{itemize}
\item The word $w$ is the maximal prefix of $x$ such that $w$ belongs to the language of $\mathbb{K}$. 
Thus we have, for some $D>0$, $d(x,\mathbb{K})=D^{-d}$ with $x=w\dots$ and $w=x_1\dots x_{d}$. Let us denote $\delta(x)=d$, and  $\delta_k^n=\delta(\sigma^k\circ H^n(x))$ for all integers $k$ and $n$. Note that $\de=\de_{0}^{0}$. 

\item If there exists an integer $b<d$ such that $\delta_b^0(x)>d-b$ and $\delta_i^0(x)=d-i$ for $i=b-1$, then we say that an {\bf accident appears at time} $b$. The {\bf depth} of the accident is $\delta_b^0$.
\end{itemize}

Remark that the word $w$ is non-empty since every letter is in the language of $\mathbb K$ if the substitution is primitive. Then, $w$ is the unique word such that 
$$x=wx', w\in\mathcal{L}_{H}, wx'_{0}\notin \mathcal{L}_{H}.$$

For a fixed $x\notin \mathbb K$, the accident times are ordered which allows to define the notion of $j^{th}$ accident with $j\ge 1$. This is done more formally in Definition  \ref{def:long-bisp-acc}.

Figure \ref{fig-accidents} illustrates the next lemma which appears in \cite{BL1}. 
\begin{lemma}\label{lem:accident-bispecial}
Let $x$ be an infinite word not in $\mathbb K$.
Assume that $\delta(x)=d$ and that the first accident appears at time $0<b\leq d$ then the word $x_b\dots x_{d-1}$ is a bispecial word of $\mathcal L_{H}$.It is called the first accident-word. 
\end{lemma}

\begin{remark}
\label{rem-acci-left-rigthspe}
If $\CA$ has cardinality two, then $x_0\dots x_{d-1}$ is not right-special. Moreover, and always if $\CA$ has cardinality two, if $x=\s(z)$ and there is an accident at time 1 for $z$, then $x_0\dots x_{d-1}$ is not left-special.
$\blacksquare$\end{remark}

\begin{center}
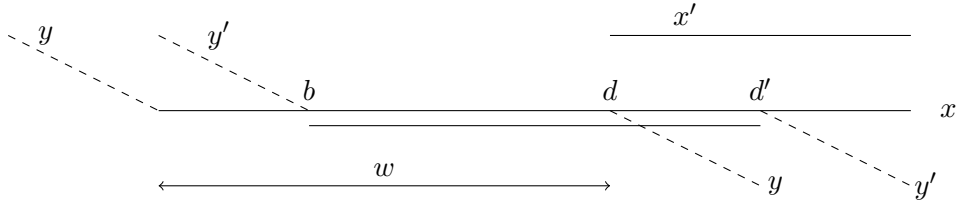
\begin{figure}[htbp]
\begin{tikzpicture}
\draw[dashed](0,1)--(2,0);
\draw(0,0)--(10,0);
\draw (3,-1) node[above]{$w$};
\draw[<->](0,-1)--(6,-1);
\draw[dashed] (-2,1)--(0,0);
\draw (-1.5,1) node{$y$};
\draw[dashed] (6,0)--(8,-1);
\draw (0.8,1) node{$y'$};
\draw[dashed] (8,0)--(10,-1);
\draw (2,-.2)--(8,-.2);
\draw (8.2,-1) node{$y$};
\draw (10.2,-1) node{$y'$};
\draw (10.5,0) node{$x$};
\draw (6,0) node[above]{$d$};
\draw (2,0) node[above]{$b$};
\draw (8,0) node[above]{$d'$};
\draw (6,1)--(10,1);
\draw (7,1) node[above]{$x'$};
\end{tikzpicture}
\caption{Accidents-dashed lines indicate infinite words in $\mathbb K$. The accident appears at $b$, the length of the accident-word is $d-b$ and the depth of the accident is $d'$. }\label{fig-accidents}
\end{figure}
\end{center}

\begin{definition}\label{def:long-bisp-acc}
We define inductiveley 
\begin{align*}
b_1&=b=\min\{j\geq 1, d(\sigma^jx,\mathbb K)\leq d(\sigma^{j-1}x,\K)\}\\
b_2&=\min\{j\geq 1, d(\sigma^{j+b_1}x,\mathbb K)\leq d(\sigma^{j+b_1-1}x,\K)\}\\
b_3&=\min\{j\geq 1, d(\sigma^{j+b_1+b_2}x,\mathbb K)\leq d(\sigma^{j+b_1+b_2-1}x,\K)\}\\
\dots
\end{align*}
Set $b_{0}=0$,  and inductively $B_j=b_0+\dots +b_j$.
Then, the integer $B_j, j\geq 1$ is the $j^{th}$ accident time and $d_{j}:=\delta(\sigma^{B_{j}}x)$ is its depth.  
The word $x_{b_j}\dots x_{d_{j-1}-1}$ is called  the {\bf $j^{th}$ accidents-word}. Its length is called the {\bf length of the $j^{th}$ accident}.
\end{definition}

\begin{remark}
\label{rem-accident0}
By convention, the $0^{th}$ accident is at time zero. 
$\blacksquare$\end{remark}

\begin{lemma}\label{lem:deux-acc}
Consider $x$ such that $\delta(x)=d$. Denote by $B_1, B_2$ the times of first and second accidents. Assume the two bispecial words defined by the accidents do not overlap, then we have:
$$\begin{cases}
\delta_i(x)=d-i, 0\leq i< B_1\\
\delta_i(x)=d'-B_1-i, B_1\le i< B_2
\end{cases}$$
\end{lemma}
\begin{proof}
It is a simple application of the definition of accident. See also Figure \ref{fig-accidents} with $B_{1}=b$.
\end{proof}

We recall that for $x\in \CA^{\N}$ of the form $x=a\ldots$  and for a primitive, 2-full and marked substitution $H$, 
we have set $t_{n}(x)=|H^{n}(a)|$. Then, we set:
\begin{definition}
\label{def-cbn}
We denote by $\CB_{n}(x)$ the set of $j^{th}$ accidents-words with $j\leq t_{n}(x)$. 
\end{definition}



\section{Proof of Theorem \ref{thm:fixe}}\label{secthm1}

\subsection{Renormalization operator and accidents}
In order to prove Theorem \ref{thm:fixe} we need to compute $R^{n}\varphi$. We give here a formula for $R^{n}\varphi(x)$ and explain why $\displaystyle\lim_{\ninf}R^{n}\varphi(x)$ only depends on the germ of $\varphi$ close to $\K$.

\subsubsection{A formula for $R^{n}\varphi$}
We emphasize that $\s$ satisfies the following renormalization equation (with respect to $H$) 
$$H\circ\sigma(x)=\sigma^{t_1(x)}\circ H(x).$$

This equality is the key point to  prove the formula that gives an expression for $R^{n}$:
\begin{lemma}\label{lem:fixecalc}
For every integer $n$ and for every $x\in\mathcal A^{\mathbb N}$ we have
$$R^n\varphi(x)= \displaystyle\sum_{i=0}^{t_n(x)-1}\varphi\circ\sigma^i\circ H^n(x).$$
\end{lemma}
\begin{proof}
We make a proof by induction:\\
For $n=1$ it is clear. Assume the result is true for $n$.
For all $j\in[0\dots t_1(H(x))-1]$, and for all $i\in[0\dots t_1(x)-1]$ we have:
$$H\circ \sigma^i= \sigma^{s(i,x)}\circ H,\quad \text{where}\quad s(i,x)=\displaystyle\sum_{j=1}^it_1(\sigma^{j-1}(x)).$$

By induction hypothesis we deduce
$$R^{n+1}\varphi(x)=R^n\circ R\varphi(x)=\displaystyle\sum_{j=0}^{t_1(x)-1}\sum_{i=0}^{ t_n(x)-1} \varphi\circ\sigma^j\circ H\circ \sigma^i\circ H^n(x)$$
 $$R^{n+1}\varphi(x)=\displaystyle\sum_{j=0}^{t_1(x)-1}\sum_{i\leq t_n(x)-1} \varphi\circ \sigma^{s(i,x)+j}\circ H^{n+1}(x)$$
 $$=\displaystyle\sum_{i=0}^{t_{n+1}(x)-1}\varphi\circ\sigma^i(x)\circ H^{n+1}(x).$$
We used the fact that
$t_{n+1}(x)=|H^{n+1}(a)|=|H(H^n(a))|=\displaystyle\sum_{i=1}^{t_n(x)}t_1(i)$.
The induction hypothesis is proved.
\end{proof}

\subsubsection{Distance between $\s^{j}(H^{n}(x))$ and $\K$}

Lemma \ref{lem:fixecalc} shows why it is so important to know the numbers  $\de_{k}^{n}(x)=\de(\s^{k}(H^{n}(x)))$ for every $x$ and for $k\le t_{n}(x)-1$. We shall see below why accidents perturb the computation of $R^{n}(\varphi)(x)$. This explains why we need to control them. 

\medskip
Moreover, $R^{n}\varphi(x)$ involves a Birkhoff sum at point $H^{n}(x)$ which changes if $n$ increases. Clearly, $H^{n}(x)$ converges to a fixed point of $H$, thus goes to $\K$ if $n$ increases. But this convergence may be faster than what we could expect, just knowing for how many digits $x$ coincides with $\K$.  We give here two  examples illustrating this point:

\paragraph{Example.}

Consider $H:\begin{cases}a\rightarrow abbaaa\\ b\rightarrow baaaab\end{cases}$. The word $bbb$ does not belong to the language. Nevertheless $H(bbb)$ belongs to $\mathcal L$ as seen by the computation of 
$$H(aaaa)=abbaaa abbaaa abbaaa abbaaa=abH(bbb)aa$$
Here, for $x=bbb\dots$ we have $\delta(x)=2$ and $\delta(H(x))=\de_{0}^{1}(x)\geq 3*6>2*6.$

\medskip

Consider  $H:\begin{cases}a\mapsto aaab\\ b\mapsto abaa
\end{cases}$. We have $H(a^3)=a^3ba^3ba^3b=a^2H(bb)ab$, thus $bb$ does not belong to the language, and $H$ is not 2-full.
Nevertheless we have $H(bb)=aba^3ba^2$, which is a factor of $H(aaa)$.
Now let $x=b\sigma^3H^\infty(a)$, then we obtain
$x=bba^3ba^3baba^5ba^3b\dots$ Remark that
$\delta(x)=1$. Moreover $H(x)=aba^3ba^5b\dots$, thus we obtain $\delta^1_0(x)=7$.

\subsubsection{Necessity of 2-full hypothesis and germ of a potential close to $\K$}
We can now explain why knowing the germ close to $\K$ is sufficient to determine $\displaystyle\lim_{\ninf}R^{n}\varphi(x)$. 
Note that $H$  is 2-full which means that for every $x$, $\de(x)\ge 2$. Set $x=ab\ldots$, it follows that $\de_{0}^{n}(x)$ is bigger than $t_{n}(a)+t_{n}(b)$, and then for every $k\le t_{n}(a)-1$
\begin{equation}
\label{eq-lowboundrnv}
\de^{k}_{n}(x)\ge t_{n}(b)+t_{n}(a)-k.
\end{equation} 
Remember that $t_{n}(b)$ is bounded by $c.\lambda^{n}$ with $c>0$. This computation shows that among all the points $\s^{k}(H^{n}(x))$, the farthest from $\K$ is at distance at most $D^{-t_{n}(b)-1}\sim D^{-\l^{n}}$.  It thus makes sense to replace $V(\s^{k}(H^{n}(x)))$ by $\disp g(\s^{k}(H^{n}(x)))/(\de_{k}^{n}(x))^{\al}$.

\bigskip
\paragraph{\bf Counter-example}\label{exemple:tm}
On the contrary, consider the following substitution

$$H=\begin{cases}a\rightarrow abba\\ b\rightarrow bab\end{cases} $$

This substitution is primitive, marked but is not 2-full since $aa$ does not belong to the language.

Then consider $x=aa\dots$ we have $\delta(x)=1$.
Therefore, $H^{n}(x)=H^{n}(a)H^{n}(a)\dots$. Note that $H^{n}(a)$ finishes and starts with $a$ and then $H^{n}(a)H^{n}(a)$ contains the word $aa$ in its middle.  Furthermore, any suffix of $H^{n}(a)$ is in the language  but no suffix of $H^{n}(a)a$ belongs to the language. Therefore,  for any $i\le n$ $\delta_i^n(x)=|H^n(a)|-i$. We will see at the end of the paper that $R^{n}(\varphi)(x)$ does not converge. This shows that knowing the germ close to $\K$ is not sufficient to determine the limit for $R^{n}(\varphi)(x)$.

\subsection{Bispecial words for marked substitutions}
As we have seen above, it is important to detect accidents. We also pointed out that accidents are related to occurrences of bispecial words in the language. It is therefore of prime importance to study these bispecial words.
We prove here a strong version of Theorem \ref{thm:mosse} in Theorem \ref{thm-mosse-plus}. This allows us to get a complete description of the set of bispecial words (see Proposition \ref{prop:bisp-form}).

\begin{lemma}\label{lem:periode}
Assume that $H$ is a marked substitution.  
If $z=H(x)=SH(y)$ is an infinite word where $S$ a finite word in $\mathcal A^{*}$ which is a strict suffix of the image of a letter by $H$. Then either $S$ is empty and $x=y$ or the word $z$ is ultimately periodic.
\end{lemma}
\begin{proof}
If $S$ is the empty word, then the left marking proves the result.
If not, then let us denote by $t$ the length of $S$. Denote $x=x_1x_2\dots$ . 
The infinite word $H(x)$ can be cut by construction into words corresponding to the images of the letters by $H$, {\it i.e} $H(x)=H(x_1)H(x_2)\dots$. Let us do the same thing for $H(y)$. Since $H$ is left marked, the first letters of the image are in bijection with the alphabet, thus we can assume that $H(x_i)$ begins with $x_i$ for every integer.
We denote by $t'=| |H(x_1)|-t|$, see Fig. \ref{fig-retaralallumage}. 
\begin{figure}[htbp]
\begin{center}
\begin{tikzpicture}
\draw (-3,0)--(5,0);
\draw (-3,-.1)--+(0,.2);
\draw (-1,-.1)--+(0,.2);
\draw (1,-.1)--+(0,.2);
\draw (3,-.1)--+(0,.2);

\draw (-3,-2)--(5,-2);

\draw (-2,-2.1)--+(0,.2);
\draw (0,-2.1)--+(0,.2);
\draw (2,-2.1)--+(0,.2);
\draw (4,-2.1)--+(0,.2);(-)

\draw[dashed] (-2,-2)--(-2,0);
\draw (-2.5,-2) node[below]{$t$};
\draw[<->] (-3,-1.8)--(-2,-1.8);
\draw[<->](-2,-0.5)--(-1,-0.5);
\draw (-1.5,0) node[below]{$t'$};
\draw(-2,0) node[above]{$H(x_{1})$};
\draw (0,0) node[above]{$H(x_2)$};
\draw (-1,-2) node[above]{$H(y_1)$};
\draw (1,-2) node[above]{$H(y_2)$};
\end{tikzpicture}

\caption{$\s^{t}H(x)=H(y)$}
\label{fig-retaralallumage}
\end{center}
\end{figure}
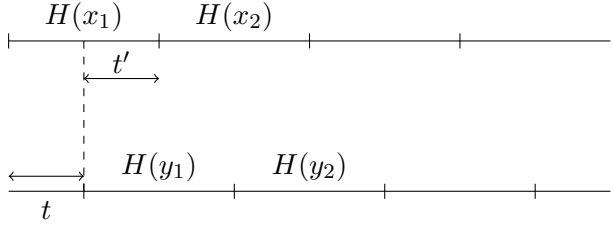

First of all assume that $t+|H(y_1)|=|H(x_1)|+|H(x_2)|$. Then we have 
$SH(y_1)=H(x_1x_2)$, the hypothesis of right marking allows us to deduce $y_1=x_2$ and $S=H(x_1)$ which is impossible.

Thus  we can define a function  $\psi$ on $\mathcal{A}^2\times[0\dots \max|H(a)|]$ by the formula
$$
\begin{array}{rcl}
\mathcal{A}^2\times[0\dots \max|H(a)|]&\rightarrow&\mathcal{A}^2\times[0\dots \max|H(a)|]\\
(x_1,y_1,t)&\mapsto&\psi(x_1,y_1,t)=\begin{cases} (x_2,y_1,t') \quad t<|H(x_1)|\\ 
(y_1,x_2,t')\quad t>|H(x_1)| \end{cases}\\
\end{array}
$$
This function is defined on a finite set and can be iterated by the previous argument, thus $\psi$ is ultimately periodic. 
This implies that the word $z$ is ultimately periodic by the pigeonhole principle.
\end{proof}

From Lemma \ref{lem:periode} we deduce a very important result. If $x$ belongs to $\CA^{\N}\setminus\K$, then so does $H(x)$:

\begin{corollary}\label{cor:fini}Consider a marked substitution $H$. For each word $x=wx'$ with $w\in\mathcal{L}_{H}$ and $wx'_{0}\notin\mathcal{L}_{H}$, for every integer $s$ there exists $m<\infty$ such that $\delta[H^s(x)]=m$
\end{corollary}
\begin{proof}
The proof is by contradiction and by induction. Assume $H(x)\in \mathbb K$ thus it can be written $SH(y)$ with $y\in\mathbb K$.  Then we apply  Lemma \ref{lem:periode}. If $S=\epsilon$ (the empty word) then, $x=y$ and it is a contradiction with our assumption. If $S\neq\epsilon$, then $y$ is ultimately periodic which is in contradiction with Theorem \ref{thm:mosse}. This shows 
$$x\notin\K\Longrightarrow H(x)\notin\K.$$
Then, the result follows by induction. 
\end{proof}

\begin{theorem}\label{thm-mosse-plus}
Consider a primitive, aperiodic and marked substitution. 
There exists $l(H)>0$ such that for every $z\in \mathcal L_H$ with $|z|>l(H)$ there exists a unique decomposition 
$z=SH(x)P$ with $(S,P)\in\mathcal S\times\mathcal P$, $S$ is a suffix of $H(s)$, $P$ is a prefix of $H(p)$ and $sxp\in \CL_{H}$. 
\end{theorem}
\begin{proof}
The existence of the decomposition is clear because $\K=\disp\ol{\{\s^{n}(v),n\in\N\}}$ where $v$ is any fixed point for $H$. 
Now assume we have two decompositions
$$SH(x)P=S'H(y)P'.$$

We will apply an effective version of the proof of Lemma \ref{lem:periode}. Let us denote $s=\max_a|H(a)|$.
The same proof can be applied, it suffices to remark that the period and the pre-period are bounded by the cardinality $D$ of the finite alphabet $\CA$. Consider the minimum $p_0$ of the integers $p$ such that $(D^2s)^p+sD^2>N_H$.
The proof is done with $l(H)=(D^2s)^{p_0}+sD^2$. We deduce $S=S'$, then the same argument shows that $P=P'$.
\end{proof}

\begin{remark}
\label{rem-thmosse+false}
We emphasize that Theorem \ref{thm-mosse-plus} is false without the marked assumption. 
Consider
$H:\begin{cases}a\rightarrow aba\\ b\rightarrow ab\end{cases}$ which is not marked. Note that both $aa$ and $ab$ belong to the language. We thus claim that  there exists a sequence of right special words with length going to infinity. Let $u$ be a right-special word with length as big as wanted.  Then we have $H(ua)=H(u)H(a)=H(u)aba=H(u)H(b)a=H(ub)a$. 
This contradicts uniqueness of the decomposition $H(ua)$. 
$\blacksquare$\end{remark}

\begin{proposition}\label{prop:bisp-form}
Let $H$ be a primitive, aperiodic and marked substitution. Let $\CW_{b}$ be the set of bispecial words of length less than $l(H)$. Then every bispecial word can be written as $H^n(v)$ with $v\in \CW_{b}$ and $n$ some integer.
\end{proposition}
\begin{proof}
Consider a bispecial word $u$. By Theorem \ref{thm-mosse-plus} we can write $u=SH(v)P$ where $v$ has maximal length, $v$, $S$ and $P$ are unique. 

We claim that $S$ is empty. Indeed, since $u$ is a bispecial word, there exist two letters such that $au$ and $bu$ belong to the language.
If $S$ is non-empty, then $aS, bS$ are the suffixes  with the same length of  $H(c)$ where $c$ is a letter (unique by assumption on $H$). 
We deduce $a=b$, which is impossible. The same argument applies for $P$.

Now we prove that $v$ is a bispecial word. If $aH(v)$ belongs to the language $\mathcal L_H$, the properties of $H$ show that it is the suffix of a unique word $H(c)H(v)$.  The same argument works for $bH(v)$ the other left extension of $H(v)$. The two left extensions of $v$ are different by assumption on $H$. By the same argument $v$ is right special. 
The proof finishes by an iteration of this process.
\end{proof}

We recall that $\l$ is the dominating eigenvalue for the incidence matrix of $H$. Then Proposition  \ref{prop:bisp-form} yields:
\begin{corollary}\label{coro:long-bisp}
There exist $0<\theta<\lambda$ and a finite set of positive numbers $c$, such that the lengths of the bispecial words of $\mathcal L_H$ are of the form $c\lambda^n+O(\theta^{n})$, $n\in\N$.
\end{corollary}
Note that the numbers $c$ are the lengths of the words in $\CW_{b}$. 

\subsection{Crucial Proposition}
By Lemma \ref{lem:fixecalc}, we have a formula for $R^{n}(\varphi)(x)$. To study the convergence of this term we need to get good estimates for $\de_{i}^{n}(x)$ for $i< t_{n}(x)$ and for any $x\notin\K$ (see also the discussion after Lemma \ref{lem:fixecalc}).  We have an easy bound from above :
$$\delta_i^n(x)\geq \delta_0^n(x)-i,$$
but we need a sharper estimate. For that purpose, we need to know the accident words $\CB_{n}(x)$ (recall \ref{def-cbn}). The following main proposition shows how accidents occur.

\begin{proposition}\label{prop:acc}
Let $H$ be a 2-full, marked, aperiodic and primitive substitution. 
Let $x\notin\K$ and $p$ be such that $\de_{0}^{0}(x)=p$. Set $x=w_{1}.\ldots w_{p}x_{p+1}\dots\notin\mathbb K$ and let $k$ be such that $|H^k(w_2\dots w_p)|\ge l(H)$. Then
$$\mathcal B_n(x)=H^{n-k}(\mathcal B_k(x))\quad \text{for }n\geq k.$$
\end{proposition}
\begin{proof}
Note that $x=wx_{p+1}\ldots$ and $w\in\CL_{H}$. Let us write $H^k(x)= e_1\dots e_{m_k}e_{m_k+1}\dots$ with $m_k=\delta^k_0(x)$. Corollary \ref{cor:fini} shows that  $m_k$ is finite.

\begin{itemize}
\item First we prove $\delta_0^n(x)=|H^{n-k}(e_1\dots e_{m_k})|$. 
Note that we have the relation $H^n(x)=H^{n-k}H^k(w_1\dots w_p\dots)=H^{n-k}(e_1\dots e_{m_k}e_{m_k+1}\dots)$, which shows that $\de_{0}^{n}(x)$ is bigger than $|H^{n-k}(e_1\dots e_{m_k})|$ because $e_{1}\ldots e_{m_{k}}$ belongs to $\CL_{H}$. 
Actually, the proof is also done by induction on $n\geq k$. 

Assume by contradiction that $\delta_0^{k+1}(x)$ is strictly bigger than the number $|H(e_1\dots e_{m_k})|$. This means that there exists a letter $a$ such that $H(e_1\dots e_{m_k})a\in \mathcal L_H$. Note that $|H(e_1\dots e_{m_k})|> |H^{k}(w_{2}\ldots w_{p})|\ge l(H)$, 
we can thus apply Theorem \ref{thm-mosse-plus} to the word $H(e_1\dots e_{m_k})a$. By the left marking of $H$ we deduce that $e_1\dots e_{m_k}e\in \mathcal L_H$ with letter $e$ such that $H(e)$ begins with $a$, as $H(e_{m_{k+1}})$. This is a contradiction with the definition of $m_k$. We then iterate this argument, noting that $|H^{j}(e_1\dots e_{m_k})|$ increases in $j$ and is thus bigger than $l(H)$.

\item Now consider the time of the first accident of $H^k(x)$ and denote it by $j_1$. We argue by contradiction and prove that $H^{n}(x)$ cannot have an accident for $i<|H^{n-k}(e_1\dots e_{j_1})|$. 
By definition we have $j_1<t_{k}(x)\leq m_k$ and $\delta_{j_1}^k(x)>m_k-j_1$ whereas $\de_{j_{1}-1}^{k}(x)=m_{k}-j_{1}+1$. 

Pick $0<i<|H^{n-k}(e_1\dots e_{j_1})|$ and assume that $\delta_i^n(x)>\delta_0^n(x)-i$. We have $H^n(x)=H^{n-k}(e_1)H^{n-k}(e_2)\dots$. Let us introduce $l$ the smallest integer such that $i<|H^{n-k}(e_1\dots e_l)|$. A prefix of $\sigma^i H^n(x)$ can be written $S H^{n-k}(e_{l+1}\dots e_{m_k})a\in \CL_{H}$ with $S$ suffix of $H^{n-k}(e_{l})$ and $a\in\CA $. Note that $l\le j_{1}<t_{k}(x)$, which yields that $H^{n}(w_{2}\ldots w_{p})=H^{n-k}(H^{k}(w_{2}\ldots w_{p}))$ is a factor of $H^{n-k}(e_{l+1}\ldots e_{m_{k}})$. 
We can thus apply Theorem \ref{thm-mosse-plus} and by the right marking of $H^k$, we obtain a word suffix of $e_l\dots e_{m_k}e\in \CL_{H}$.  This means that $H^{k}(x)$ has an accident at time $l-1<j_{1}$ and this is a contradiction with  the definition of $j_1$.  
Finally we have proven $$\delta_i^n(x)=\delta_0^n(x)-i, 0\leq i\leq |H^{n-k}(e_1\dots e_{j_1})|-1.$$

\item By definition of an accident we know that $e_{j_1}\dots e_{m_k}e\in\mathcal L_H$ for some letter $e$. Then by application of $H^{n-k}$ we deduce that there exists some letter $a$ such that $H^{n-k}(e_{j_1}\dots e_{m_k})a\in\mathcal L_H$. Thus the first accident of $H^n$ appears at time $|H^{n-k}(e_1\dots e_{j_1})|$. The same reasoning shows that the accident-word is the image by $H^{n-k}$ of the first accident-word of $H^k$.

\item Let us denote by $j_2$ the time of the second accident of $H^k(x)$. Note that $H^{n}(w_{2}\ldots w_{p})$ has length bigger than $l(H)$ and is still a factor of $H^{n-k}(e_{j_{2}}\ldots e_{m_{k}})$  because $j_{2}<t_{k}(x)$. Note also that $\s^{j_{1}}(H^{k}(x))$ coincides with a word of $\K$ for at least $m_{k}-j_{1}+1$ digits. In other words, $H^{n-k}(e_{j_{1}}\ldots e_{m_{k}}e_{m_{k+1}})$ is a suffix of the coincidence of $\s^{j_{1}}(H^{n}(x))$ coincides with $\K$. This suffix contains $H^{n-k}(e_{j_{2}}\ldots e_{m_{k}})$, thus it also contains $H^{n}(w_{2}\ldots w_{p})$. We can thus repeat the same process  to $j_{2}$ and more generally to each accident of $H^k(x)$.
\end{itemize}
\end{proof}

\begin{corollary}\label{cor-accidentHkx}
Denote the times of accidents of $H^k(x)$ by $j_1, j_2,\dots j_s$, and their depths by $\Delta_{j_1},\dots, \Delta_{j_s}$. 
We have:
\begin{itemize}
\item The accidents of $H^n(x)$ appear at times $t_{i,n-k}:=\lambda^{n-k}j_i+O(\theta^{n-k}), i\leq s$.

\item Their depths are equal to $\Delta_{i,n-k}:=\lambda^{n-k}\Delta_{j_i}+O(\theta^{n-k}), i\leq s.$

\end{itemize}
where  $0<\theta<\lambda$.
\end{corollary}
\begin{proof}
This a a direct corollary of the previous proposition and Corollary \ref{coro:long-bisp}. Note that $\Delta_{i,0}=\Delta_{j_{i}}$. 
\end{proof}

\subsection{Proof of Theorem \ref{thm:fixe}}
\subsubsection{Preliminary lemma}

\begin{lemma}\label{lem:riemann}
Let $a,\lambda$ be some positive real numbers and $f$ a Lipschitz function defined on  a neighborhood of $[0,a]$. Let $\phi:\mathbb N\rightarrow \mathbb R$ be a real sequence such that $|\phi(n)|\leq C\theta^n$ with $C>0$ and $0<\theta<\lambda$.
We have
$$\displaystyle\lim_{\ninf}\frac{1}{\lambda^n}\sum_{k=0}^{[a\lambda^n]}f\left(\frac{k+\phi(n)}{\lambda^n}\right)=\int_0^af(x)dx.$$
\end{lemma}
\begin{proof}
Let us denote $S_n$ the sum and $K$ the Lipschitz constant of the function $f$. We obtain
$$\left|S_n-\frac{1}{\lambda^n}\sum_{k=0}^{[a\lambda^n]}f\left(\frac{k}{\lambda^n}\right)\right|\leq \frac{1}{\lambda^n}\displaystyle\sum_{k=0}^{[a\lambda^n]}\left|f\left(\frac{k+\phi(n)}{\lambda^n}\right)-f(\frac{k}{\lambda^n})\right|$$
$$\leq \frac{1}{\lambda^n}a\lambda^n.K.\frac{|\phi(n)|}{\lambda^n}\leq Ka\frac{|\phi(n)|}{\lambda^n}.$$
The upper bound converges to zero as $n$ goes to infinity. The term $\frac{1}{\lambda^n}\displaystyle\sum_{k=0}^{a\lambda^n}f\left(\frac{k}{\lambda^n}\right)$ is a Riemann sum, thus we deduce the result.
\end{proof}

\begin{remark}
\label{rem-f-unifco}
The same type of proof works if $f$ is an uniformly continuous function. It also holds if the sum is done up to $a\l^{n}+o(\l^{n})$ instead of $a\l^{n}$. 
$\blacksquare$
\end{remark}

\subsubsection{Computation of $\displaystyle\lim_{\minf}R^{m}\varphi$: the case $g\equiv 1$}
We want to compute $\displaystyle\lim_{\minf} R^m(\varphi)$. By Lemma \ref{lem:fixecalc} we have
$$R^m\varphi(x)= \displaystyle\sum_{i=0}^{t_m(x)-1}\varphi\circ\sigma^i\circ H^m(x).$$
The potential $\varphi$ has the following form  $\varphi(x)=\frac{1}{p^\alpha}+o(\frac{1}{p^{\al}})$.

$\bullet$ First of all consider the case $\alpha=1$.
Since $\varphi(x)=\frac{1}{p}+o(\frac{1}{p})$ if $\delta(x)=p$, we obtain
$$R^m\varphi(x)=\displaystyle\sum_{j=0}^{t_m(x)-1}\frac{1}{\delta_j^m(x)}+o(\frac{1}{\delta_j^m}).$$
We emphasize that the term  $o(\ldots)$ is actually a negligible term with respect to the first summand. 
Therefore, it does not influence the limit for $R^{m}\varphi(x)$ and we shall forget it in the rest of our proof.

\bigskip
We pick some $x\notin\K$ and reemploy notations from Corollary \ref{cor-accidentHkx}. 
Let $p=\de(x)$ and $k$ be such that $|H^{k}(x_{2}\ldots x_{p})|>l(H)$. Let
$j_1, j_2,\dots j_s$ be the times of accidents of $H^k(x)$, $\Delta_{j_{i}}$ their corresponding depths.  
The accidents of $H^m(x)$ appear at times $t_{i,m-k}:=\lambda^{m-k}j_i+O(\theta^{m-k})$ with depths 
$\Delta_{i,m-k}=\lambda^{m-k}\Delta_{j_{i}}+O(\theta^{m-k})$.

Moreover, by  Lemma \ref{lem:deux-acc}  
$$\delta_j^m(x)=\Delta_{i,m-k}-(j-t_{i,m-k})\quad t_{i,m-k}\leq j< t_{i+1,m-k}$$
holds.

We split the sum $\disp \sum_{j=0}^{t_{m}(x)-1}$ into the sums $\disp\sum_{j=t_{i,m-k}}^{t_{i+1,m-k}-1}$ with the convention $t_{0,m-k}=0$ and $t_{s+1,m-k}=t_{m}(x)$. To make notations consistent we also set $j_{0}=0$, $\Delta_{0}=\de_{0}^{k}(x)$ and $j_{s+1}=t_{k}(x)-1$. Then we have

\begin{eqnarray*}
R^m\varphi(x)&=&\displaystyle\sum_{l=0}^{t_{1,m-k}-1}\frac{1}{\Delta_{0,m-k}-l}+\displaystyle\sum_{l=t_{1,m-k}}^{t_{2,m-k}-1}\frac{1}{\Delta_{1,m-k}-l+t_{1,m-k}}\\
&&+\dots
+ \displaystyle\sum_{l=t_{s,m-m}}^{t_{m}(x)-1}\frac{1}{\Delta_{s,m-k}-l+t_{s,m-k}}+o(\ldots)\\
&&=\sum_{i=0}^{s}\sum_{l=0}^{t_{i+1,m-k}-t_{i,m-k}-1}\frac1{\Delta_{j_{i}}\l^{m-k}-l+\phi_{i}(m-k)}\\
&&=\sum_{i=0}^{s}\sum_{l=0}^{(j_{i+1}-j_{i})\l^{m-k}+\phi'_{i}(m-k)}\frac1{\Delta_{j_{i}}\l^{m-k}-l+\phi_{i}(m-k)},
\end{eqnarray*}
where $\phi_{i}(m-k)$ and $\phi'_{i}(m-k)$ are in $O(\theta^{m-k})$ with $0<\theta<\l$. 
The computation of the sums is made with Lemma \ref{lem:riemann}. We  finally obtain
$$U(x)=\displaystyle\lim_{+\infty} R^m\varphi(x)=\displaystyle\sum_{i=0}^{s}\log{\left(\frac{\Delta_{j_i}}{\Delta_{j_i}-(j_{i+1}-j_{i})}\right)}.$$
Note that this last quantity only depends on how close $H^{k}(x)$ is to $\K$. This shows that $U$ is continuous. 

$\bullet$  It remains to consider the cases $\alpha\neq 1$. The proof is simpler and is based on convergence of Riemann sums. In all the cases, the renormalization term to get a Riemann sum is $\l^{-\al(m-k)}$ and the sums have $\l^{m-k}$ summands. 
For $\al>1$, the renormalization term is too heavy and the sum goes to 0. For $\al<1$ the renormalization term is too light and the sum goes to $+\8$. 
We left the exact computations to the reader and refer to \cite{BL1, BL2} for similar computations.

\subsubsection{Limit for $R^{m}\varphi(x)$. The general case}\label{fin-thm1}
We consider $\varphi$ of the form $\varphi(x)=\disp\frac{g(x)}{p^{\al}}+o(\frac1{p^{\al}})$ if $\de(x)=p$ and with $g$ a positive and continuous function. 
First, we emphasize that continuity and positiveness for $g$ imply that $g$ is bounded from above and from below away from zero. Therefore, the proof for $\al\neq 1$ is the same. We can thus focus on $\al=1$.

In that case we need to compute 
$$R^m\varphi(x)=\displaystyle\sum_{j=0}^{t_m(x)-1}\frac{g\circ \s^{j}(H^{m}(x))}{\delta_j^m(x)}+o(\ldots).$$

There are two main arguments to deal with these extra terms. First, we show that the terms $g\circ\s^{j}(H^{n}(x))$ can be exchanged by terms $g\circ \s^{k}(H^{n}(y_{k,j}))$ with $y_{k,j}\in\K$. Then, we use  a technical lemma to show the convergence to the desired quantity.

\paragraph{Replacing $g\circ \s^{j}(H^{n}(x))$.}
We reemploy notations from above. Let $j_{1},\ldots j_{s}$ the times of accidents for $H^{k}(x)$, We also set $j_{0}=0$ and $j_{s+1}=t_{k}(x)-1$. We have defined $t_{i,m-k}$ and $\Delta_{i,m-k}$. 

There exist points $y^{0},\ldots, y^{s}$ in $\K$ such that $d(\s^{j_{i}}(H^{k}(x)),y^{i})=d(\s^{j_{i}}(H^{k}(x)),\K)$. In other words, the $y^{i}$'s are points in $\K$ and coincide with $\s^{j_{i}}(H^{k}(x))$ for exactly $\de_{j_{i}}^{k}(x)$-digits.

Now, we refer the reader to Figure \ref{fig-hm-krenor} for the next discussion. 
We claim that Proposition \ref{prop:acc} implies that for every $m\ge k$, for every $t_{i,m-k}\le j<t_{i+1,m-k}$ 
\begin{equation}
\label{eq-distance-hm-k}
\de_{j}^{m}(x)=d(\s^{j}(H^{m}(x)),\K)=d(\s^{j}(H^{m}(x)),H^{m-k}(y^{i})).
\end{equation}
As $H$ is 2-full, for every $i$, $\de_{j_{i}}^{k}(x)\ge j_{i+1}-j_{i}+1$ (otherwise $j_{i+1}-1$ would be an accident) and then for $0\leq j\leq t_{i+1,m-k}-t_{i,m-k}$
\begin{equation}
\label{eq-majohm-khm}
d(\s^{t_{i,m-k}+j}(H^{m}(x)),\s^{j}(H^{m-k}(y^{i})))=D^{-\Delta_{i,m-k}+j}\le D^{-\l^{m-k}+O(\theta^{m-k})}. 
\end{equation}

\begin{center}
\begin{figure}[htbp]
\begin{tikzpicture}\label{fig-hm-krenor}
\draw[dashed](0,1)--(2,0);
\draw(0,0)--(10,0);
\draw[dotted] (5,1)--(7,0);
\draw (5,1) node{$y^{i+1}$};
\draw (0.8,1) node{$y^{i}$};
\draw[dashed] (8,0)--(10,-1);
\draw (9.2,-1) node{$y^{i}$};
\draw (10.5,0) node{$H^{k}(x)$};
\draw (2,0) node[below]{$j_{i}$};
\draw[<->](2,0.2)--(8,0.2);
\draw (7,0) node[below]{$j_{i+1}$};
\draw (5,0.2) node[above]{$\Delta_{i}$};
\draw[<->](7,0.5)--(8,0.5);
\draw (9,0.5) node[above]{at least one digit};

\draw[dashed](-0.70,-3)--(0.8,-4);
\draw(-0.7,-4)--(12,-4);
\draw[dotted] (6.5,-3)--(8.5,-4);
\draw (-0.7,-3) node{$H^{m-k}(y^{i})$};
\draw[dashed] (10,-4)--(12,-5);
\draw (12.2,-5) node{$H^{m-k}(y^{i})$};
\draw (12.5,-4) node{$H^{m}(x)$};
\draw (0.8,-4) node[below]{$t_{i,m-k}$};
\draw (8.5,-4) node[below]{$t_{i+1,m-k}$};
\draw (5,-3.8) node[above]{$\Delta_{i,m-k}$};
\draw[<->](8.5,-3.5)--(10,-3.5);
\draw (9,-3.5) node[above]{$\ge\l^{m-k}$ digits};
\draw[<->](0.8,-3.8)--(10,-3.8);
\draw[<-, loosely dashed](0.8,-3.7)--(2,-0.5);
\draw(1.8,-2)node{$H^{m-k}$};
\draw[->, loosely dashed](5,-0.5)--(5,-3.2);
\draw(5.2,-2)node{$H^{m-k}$};
\draw[->, loosely dashed](10.5,-0.5)--(12.5,-3.7);
\draw(11,-2)node{$H^{m-k}$};

\end{tikzpicture}
\caption{$H^{m-k}$ renormalization}
\end{figure}
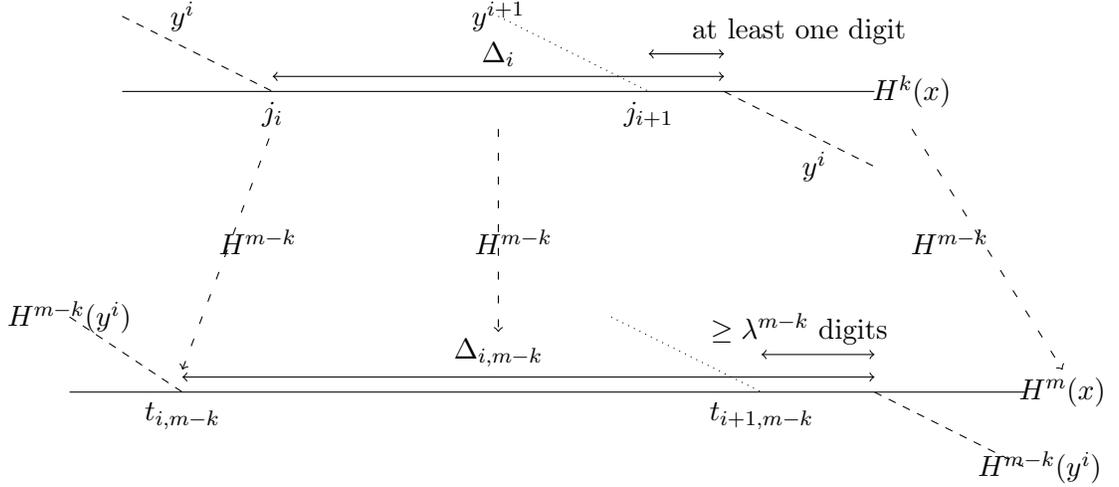
\end{center}
 This shows that replacing $\s^{j}(H^{m}(x))$ by $\s^{j-t_{i,m-k}}(H^{m-k}(y^{i}))$ for $t_{i,m-k}\le j<t_{i+1,m-k}$ just add an error in $o(D^{-\lambda^{m-k}})$ and thus does not influence the limit. 
 Then we have 
 \begin{eqnarray*}
R^{m}\varphi(x)&=& \sum_{j=0}^{t_{m}(x)-1}\frac{g\circ\s^{j}(H^{m}(x))}{\delta_{j}^{m}(x)}+o(\ldots)\\
&=& \sum_{i=0}^{s}\sum_{l=0}^{t_{i+1,m-k}-t_{i,m-k}-1}\frac{g\circ\s^{l}\circ\s^{t_{i,m-k}} H^{m}(x)}{\Delta_{i,m-k}-l}+o(\ldots)\\
&=& \sum_{i=0}^{s}\sum_{l=0}^{t_{i+1,m-k}-t_{i,m-k}-1}\frac{g\circ\s^{l} H^{m-k}(y^{i})}{\Delta_{i,m-k}-l}+o(\ldots). 
\end{eqnarray*}

\begin{lemma}
Let $(X,\sigma)$ be an uniquely ergodic subshift.
Let $f$ be a continuous integrable function on $(0,1)$, let $g:X\rightarrow \mathbb R$ be a continuous function on $X$.
Then we have uniformly in $x\in X$:
$$\lim_{+\infty}\frac{1}{n}\displaystyle\sum_{k=0}^nf(\frac{k}{n})g(\sigma^kx)=\int_0^1f(x)dx\int_Xgd\mu.$$
\end{lemma}
\begin{proof}
Let us define $a_k=f(\frac{k}{n})$ and the Birkhoff sum $S_n(x)=\displaystyle\sum_{k=0}^{n-1}g(\sigma^kx)$ with $S_0=0$.
Finally denote $X_n=\frac{1}{n}\displaystyle\sum_{k=0}^nf(\frac{k}{n})g(\sigma^kx)$.
We have 
$$X_n=\frac{1}{n}\displaystyle\sum_{k=0}^na_k(S_{k+1}(x)-S_k(x))=\frac{1}{n}[\displaystyle\sum_{k=1}^{n+1}a_{k-1}S_k(x)-\displaystyle\sum_{k=0}^na_kS_k(x)]$$
$$X_n=\frac{1}{n}\sum_{k=1}^n(a_{k-1}-a_k)S_k(x)+\frac{a_nS_{n+1}(x)-a_0S_0}{n}$$
Now by unique ergodicity we have $\displaystyle\lim_{n\to+\infty}\frac{S_n(x)}{n}=\int_Xg(x)d\mu$ uniformly in $x$. 
Thus for all $\varepsilon>0$, there exists $N$ such that for $n\geq N$ we have $S_n(x)=n\int_Xgd\mu+n\varepsilon(n)$ 
with $\varepsilon(n)\leq\varepsilon$.

First of all assume $f\in \mathcal C^1([0,1])$. 
$$X_n=\frac{1}{n}\sum_{k=1}^n(a_{k-1}-a_k)S_k(x) +\frac{a_nS_{n+1}(x)-a_0S_0}{n}$$
$$X_n=\frac{1}{n}\sum_{k=1}^n(a_{k-1}-a_k)(k\int_Xgd\mu+k\varepsilon(k)) +\frac{a_nS_{n+1}(x)-a_0S_0}{n}$$

$$X_n=\frac{1}{n}\sum_{k=1}^{n-1}a_{k}\int_Xgd\mu-\frac{a_0+na_n}{n}\int_Xgd\mu+
\frac{1}{n}\sum_{k=1}^n(a_{k-1}-a_k)k\varepsilon(k)+\frac{a_nS_{n+1}(x)-a_0S_0}{n}$$

$$X_n=\frac{1}{n}\sum_{k=1}^{n-1}a_{k}\int_Xgd\mu+
\frac{1}{n}\sum_{k=1}^n(a_{k-1}-a_k)k\varepsilon(k)+a_n(\frac{S_{n+1}(x)}{n}-\int_Xgd\mu)-\frac{a_0S_0}{n}-\frac{a_0}{n}$$

Then there exists $c_k\in[\frac{k-1}{n},\frac{k}{n}]$ such that $a_{k}-a_{k-1}=\frac{f'(c_k)}{n}$. 
Now by property of $f$, there exists $c_k$ such that $a_{k}-a_{k-1}=\frac{f'(c_k)}{n}$

$$X_n=\frac{1}{n}\sum_{k=1}^{n-1}a_{k}\int_Xgd\mu+
\frac{1}{n^2}\sum_{k=1}^nf'(c_k)k\varepsilon(k)+a_n(\frac{S_{n+1}(x)}{n}-\int_Xgd\mu)-\frac{a_0S_0}{n}-\frac{a_0}{n}$$

We deduce there exists a constant $C>0$ such that
$$|\frac{1}{n^2}\displaystyle\sum_{k=1}^nf'(c_k)k\varepsilon(k)|\leq\frac{1}{n^2}\sum_{k=1}^NCk\varepsilon(k)+\frac{n^2-N}{n^2}\varepsilon\leq C\varepsilon $$
Thus $X_n$ converges to $\int_0^1f(t)dt\int_Xgd\mu$ uniformly in $x$.

Now if $f$ is only a continuous function, it is a uniform limit of $\mathcal C^1$ functions. We apply the previous proof.
\end{proof}

\begin{corollary}
We consider $\varphi$ of the form $\varphi(x)=\disp\frac{g(x)}{p^{\al}}+o(\frac1{p^{\al}})$ if $\de(x)=p$ and with $g$ a positive and continuous function. 
Then we have
$$\lim_{+\infty}R^m\varphi(x)=\int_{\mathbb K}gd\mu.\displaystyle\sum_{i=0}^{s}\log\left(\frac{\Delta_{j_{i}}}{\Delta_{j_{i}}-(j_{i+1}-j_{i})}\right).$$
\end{corollary}
\begin{proof}
We apply the previous lemma to $H^n(x)$, which is possible due to the uniform convergence, and use the computation in the case $g\equiv 1$.
\end{proof}


\subsubsection{Back to 2-full assumption}
We gave an example above (see page \pageref{exemple:tm}) where the substitution is not 2-full. We can now complete this example and check that for any $m$, $$R^{m}\varphi(x)=\sum_{k=1}^{|H^{m}(a)|-1}\frac1k$$ which diverges.

We emphasize that the 2-full assumption is important to guaranty some fast convergence to $\K$ iterating $H^{m}$ and taking the images by $\s^{j}$. For instance, we used the assumption in the previous proof to check that $\Delta_{i}-j_{i+1}$ is positive, which is a crucial point to exchange the $\s^{j}(H^{m}(x))$ by the $\s^{j}(H^{m-k}(y^{i}))$.

\section{The Thue-Morse substitution: example with explicit computations}\label{sec:tm}
Consider the Thue-Morse substitution
$H:\begin{cases}0\mapsto 01\\ 1\mapsto 10\end{cases}$

For this example we rephrase the proof of  Theorem \ref{thm:fixe} and give an explicit form for the potential $U$.

\begin{theorem}\label{thmtm}
For the Thue-Morse substitution there exists a unique function $U$  such that for all $x\in\mathcal A^{\mathbb N}$ we have $U(x)=\displaystyle\lim_{m} R^m\varphi(x)$ for all potential $\varphi:\mathcal{A} ^{\mathbb N}\rightarrow \mathbb{R}$ such that $\varphi(x)=\frac{1}{p}+o(\frac{1}{p})$ if $d(x,K)=2^{-p}$. 
Moreover if we denote $p=\delta(x)$ we obtain
$$U(x)=\begin{cases}\ln({\frac{p}{p-1})}\quad p\geq 3\\ \frac{1}{2}\ln{(\frac{4}{3})} \quad p=2 \end{cases}$$
\end{theorem}

\subsection{Technical lemmas}

\begin{lemma}\label{lemcass}
The Thue-Morse substitution and its language $\mathcal L$ fulfill:
\begin{itemize}
\item The fixed point which begins by $0$ can be written $$u=01.10.10.01.10.01.01.10.10.01.01\dots$$
\item The language contains the words $\begin{cases}0,1\\
00,01,10,11\\
001,010,011,100,101,110\\
\end{cases}$
\item $H$ is 2-full and marked.
\item The non uniquely desubstituable words of $\mathcal L$ are $010,101, 0101, 1010$.
\item Every word of length at least $5$ in $\mathcal L$ is uniquely desubstituable inside the language.

\end{itemize}
\end{lemma}
\begin{proof}
We refer to \cite{Pyth.02} and \cite{Cass.94} for these classical results.
\end{proof}

Let $x$ be an infinite word outside $\mathbb {K}$ which begins by a word $w$ of the language. We can always {\bf assume} that $x=w1\dots$
We denote $x=w_1\dots w_p1\dots$ where $p=\delta(x)\geq 2$. We obtain
$$H^n(x)=H^n(w_1)\dots H^n(w_p)H^n(1)\dots$$
Let us consider different cases:
\paragraph{First case:  $p\geq 3$}
\begin{proposition}
For all infinite word $x$ with $\delta(x)\geq 3$ we have
$$\delta(\sigma^k\circ H^n(x))=p2^n-k,$$
for all $k\in[0,2^n-1]$.
\end{proposition}
\begin{proof}
$\bullet$ We begin by the case $k=0$: 
The substitution has constant length, thus the length of $H^n(w)$ is equal to $p2^n$, thus we have $\delta_0\geq p2^n$.
Remark that $H^n(x)=H^{n-1}(H(w))H^n(1)\dots$, The word $H(w)$ belongs to $\mathcal{L}$ and its length is equal to $2p>4$. Assume $\delta_0>p2^n$, then  $H(w)1\in\mathcal{L}$ by Lemma \ref{lemcass}. We deduce $w1\in\mathcal L$: this yields a contradiction. Thus we have 
$\delta_0^n=p2^n.$

$\bullet$ Assume $1\leq k\leq 2^{n-1}-1$. Let us denote $H(w)=u_1\dots u_{2p}$.
We have
$$\sigma^k(H^n(x))=\sigma^kH^{n-1}(u_1).H^{n-1}(u_2\dots u_{2p})H^{n-1}(1)\dots $$
First of all remark that $\sigma^k(H^n(x))$ begins with a strict suffix of $H^{n-1}(u_1)$.
We know that $\delta(\sigma^k(H^n(x))\geq p2^n-k$. 

Assume that the word $\sigma^kH^{n-1}(u_1).H^{n-1}(u_1\dots u_{2p})1$ belongs to $\mathcal{L}$. 
We apply Lemma \ref{lemcass} with the remark that the word $\sigma^kH^{n-1}(u_1)$ is non empty and that $p\geq 3$, thus we have $2p-1\geq 5$. We deduce that $w1$ belongs to the language: contradiction. 
Thus we obtain $\delta_k^n=p2^n-k$.

$\bullet$ Now assume $k=2^{n-1}+l$ with $0\leq l<2^{n-1}$, then we have
$$\sigma^kH^n(x)=\sigma^l(H^{n-1}(u_2)).H^{n-1}(u_3\dots u_{2p})H^{n-1}(10)\dots$$
The shift acts at most on the image of $u_2$.
We know $\delta_k^n\geq p2^n-k$, and $|u_3\dots u_{2p}|=2p-2>3$. The same argument goes on:
If $H^{n-1}(u_2\dots u_{2p})1$ belongs to $\mathcal{L}$, the same is true for $u_2u_3\dots u_{2p}1$. It is equal to $u_2H(w_2\dots w_p)1$, by Lemma \ref{lemcass} since $2p-1\geq 3$. Thus it is the unique suffixe of $H(w_1w_2\dots w_p)1$: contradiction. We deduce that $\delta_k^n=p2^n-k$.
\end{proof}

\paragraph{Second case:  $p< 3$}

First of all the case $p=1$ is impossible, because the substitution is 2-full. By Lemma \ref{lemcass} the word $w$ is not right special thus it is equal either to $11$ or to $00$. The word $001$ belongs to $\mathcal L$, thus the only possibility is $w=11$ (and $111\notin \mathcal L$).

\begin{proposition}
Let $x$ be an infinite word with $\delta(x)\leq 2$, we obtain
$$\delta(\sigma^k\circ H^n(x))=\begin{cases}2.2^n-k\quad k<2^{n-1}\\ 
2^{n+1}-l \quad k=2^{n-1}+l, 0\leq l\leq 2^{n-1}-1\end{cases} .$$
Thus there is an accident.
\end{proposition}
\begin{proof}
The argument before the proof shows that $x=111\dots$

$\bullet$ First assume $k=0$.
We have $$H^n(x)=H^n(1)H^n(1)H^n(1)\dots$$
$$=H^{n-1}(1010)H^{n-1}(10)\dots$$ 
Remark that $\delta_0^n\geq 2.2^n$. 
Assume that $H^n(11)1$ belongs to $\mathcal{L}$. The word $1010$ has length $4$, we apply Lemma \ref{lemcass}, we deduce that $10101$ belongs to $\mathcal{L}$. Since $10101=H(11)1$ we deduce that $111$ belongs also to $\mathcal{L}$: contradiction. We have proved $\delta_0^n=2.2^n=2^{n+1}.$

$\bullet$ Now assume $1\leq k<2^{n-1}$, then we have $$\sigma^kH^n(x)=\sigma^k(H^{n-1}(1010))H^n(1)\dots$$
$$\sigma^kH^n(x)=\sigma^k[H^{n-1}(1)]H^{n-1}(010)H^{n}(1)\dots$$
We prove by contradiction that $\delta_k^n=2^{n+1}-k$. Since $k<2^{n-1}$ the last letter of $H^{n-1}(1)$ is not shifted by $\sigma$: we denote it $a$. The word $aH^{n-1}(010)1$ belongs to the language. Once again we apply Lemma \ref{lemcass}, we deduce $a'0101\in\mathcal L$: contradiction whatever the value of $a$ is.

$\bullet$ Now assume $k=2^{n-1}$. We obtain
$$\sigma^kH^n(x)=H^{n-1}(010)1..$$ 
The word $0101$ belongs to the language, thus we obtain $\delta_{2^{n-1}}^n\geq 2^{n+1}$. There is an accident.
Assume $\delta_{2^{n-1}}^n>2^{n+1}$. This implies that $H^{n-1}(0101)0$ also belongs to $\mathcal L$, and the same for $01010$: contradiction since $01010=H(00)0=0H(11)$. Thus we have 
$\delta_{2^{n-1}}^n=2^{n+1}.$

$\bullet$The last case is identical and left to the reader:
For $k=2^{n-1}+l$, we obtain $\delta_k^n=2^{n+1}-l$.

\end{proof}

\subsection{Proof of Theorem \ref{thmtm}}

Consider $\varphi(x)=\frac{1}{p}+o(1/p)$ with $d(x,\mathbb{K})=2^{-p}$.

\begin{itemize}
\item If $p\leq 2$ the last proposition shows:
$$R^n\varphi(x)=2\displaystyle\sum_{k=0}^{2^{n-1}-1}\frac{1}{2.2^n-k}=\frac{1}{2^{n-2}}\displaystyle\sum_{k=0}^{2^{n-1}-1}\frac{1}{4-k/2^{n-1}}$$ 
It converges to $\frac{1}{2}\int_{0}^1\frac{dx}{4-x}=\frac{1}{2}\ln{(\frac{4}{3})}$.

\item If $p\geq 3$, then we deduce

$$R^n\varphi(x)=\displaystyle\sum_{k=0}^{2^{n}-1}\frac{1}{p.2^n-k}=\frac{1}{2^{n}}\displaystyle\sum_{k=0}^{2^{n}-1}\frac{1}{p-k/2^{n}}$$ 
It converges to $\ln({\frac{p}{p-1})}$.
\end{itemize}

Finally, with the notation $p=\delta(x)$, the limit is equal to:
$$U(x)=\begin{cases}\ln({\frac{p}{p-1})}\quad p\geq 3\\ \frac{1}{2}\ln{(\frac{4}{3})} \quad p=2 \end{cases}$$

\bibliography{biblio-pression}

\end{document}